\documentclass[11pt,leqno]{amsart}

\usepackage[]{hyperref}
\hypersetup{
    colorlinks=true,       
    linkcolor=red,          
    citecolor=blue,        
    filecolor=magenta,      
    urlcolor=cyan           
}

\usepackage{amsmath}
\usepackage{amsfonts,amssymb}
\usepackage{enumerate}
\usepackage{mathrsfs}
\usepackage{amsthm}

\theoremstyle{plain}
\newtheorem{theorem}{Theorem}[section]
\newtheorem*{theo*}{Theorem}
\newtheorem{proposition}[theorem]{Proposition}
\newtheorem{lemma}[theorem]{Lemma}
\newtheorem{corollary}[theorem]{Corollary}
\newtheorem{definition}[theorem]{Definition}
\theoremstyle{definition}
\newtheorem{remark}[theorem]{Remark}
\DeclareMathOperator{\cnx}{div}
\DeclareMathOperator{\cn}{div}

\DeclareMathOperator{\diff}{d}
\DeclareSymbolFont{pletters}{OT1}{cmr}{m}{sl}
\DeclareMathSymbol{s}{\mathalpha}{pletters}{`s}

\def\ah{\arrowvert_{y=h}}

\def\ba{\begin{align}}
\def\bad{\begin{aligned}}
\def\be{\begin{equation}}
\def\ea{\end{align}}
\def\ead{\end{aligned}}
\def\ee{\end{equation}}
\def\e{\eqref}

\def\dsigma{\diff \! \sigma}

\def\dt{\diff \! t}
\def\dx{\diff \! x}

\def\dydx{\diff \! y \diff \! x}

\def\fract{\frac{\diff}{\dt}}
\def\fractt{\frac{\diff^2}{\dt^2}}

\def\cnxy{\cn_{x,y}}

\def\defn{\mathrel{:=}}

\def\eps{\varepsilon}

\def\la{\left\vert}
\def\lA{\left\Vert}
\def\le{\leq}

\def\mez{\frac{1}{2}}
\def\partialx{\nabla}

\def\ra{\right\vert}
\def\rA{\right\Vert}

\def\xN{\mathbf{N}}
\def\xR{\mathbf{R}}

\def\xT{\mathbf{T}}

\numberwithin{equation}{section}

\begin{document}

\setlength{\baselineskip}{5mm}
\title{}
\pagestyle{plain}

\date{}
\begin{center}
{\Large{\textbf{Convexity and the Hele-Shaw equation} }}
\vspace{5mm}

{
{\large{Thomas Alazard}}\\
\noindent {\small CNRS and \'Ecole Normale Sup\'erieure Paris-Saclay}}
\end{center}

\vspace{10mm}
\begin{abstract}
Walter Craig's seminal works on the water-waves problem established the importance of several exact identities: 
Zakharov's hamiltonian formulation, shape derivative formula for the Dirichlet to Neumann operator, normal forms transformations. 
In this paper, we introduce several identities for the Hele-Shaw equation 
which are inspired by his nonlinear approach. Firstly, we study 
convex changes of unknowns and obtain a large class of strong Lyapunov functions; 
in addition to be non-increasing, these Lyapunov functions  are convex functions of time. 
The analysis relies on a compact elliptic formulation of the Hele-Shaw equation, which is of independent interest. 
Then we study the role of convexity to control 
the spatial derivatives of the solutions. We consider the 
evolution equation for the Rayleigh-Taylor coefficient~$a$ (this is a positive function proportional 
to the opposite of the normal derivative of the pressure at the free surface). 
Inspired by the study of entropies for elliptic or parabolic equations, we consider the special function $\varphi(x)=x\log x$ and find 
that $\varphi(1/\sqrt{a})$ is a sub-solution of a well-posed equation.
\end{abstract}

\maketitle

\section{Introduction}

\subsection{The Hele-Shaw equation} Consider an incompressible 
liquid having a free surface given as a graph, so that, at time $t\ge 0$,  
the fluid domain is of the form
$$
\Omega(t)=\{ (x,y) \in \xT^{n}\times \xR\,;\, y < h(t,x)\},
$$
where $\xT^n$ denotes a $n$-dimensional torus, $x$ (resp.\ $y$) is the horizontal (resp.\ vertical) space variable. 
In the Eulerian coordinate system, in addition to the the free surface elevation~$h$, the 
unknowns are the velocity field 
$v\colon \Omega\rightarrow \xR^{n+1}$ and the scalar pressure $P\colon\Omega\rightarrow \xR$. We assume that they satisfy the Darcy's equations:
\be\label{Darcy}
\cnxy v=0 \quad\text{ and }\quad v=-\nabla_{x,y} (P+gy) \quad \text{in }\Omega.
\ee
A timescale may be chosen so that the acceleration of gravity is $g=1$.

These equations are supplemented by two boundary conditions. Firstly, one assumes that the pressure vanishes on the free surface:
$$
P=0\quad \text{on}\quad\partial\Omega.
$$
The second boundary condition states that the normal velocity 
of the free surface is equal to the normal component of the fluid velocity on the free surface. 
It follows that
\be\label{HS3}
\partial_t h=\sqrt{1+|\partialx h|^2} \, v\cdot n,
\ee
where $\nabla=\nabla_x$ and $n$ is the outward unit normal to $\partial\Omega$, given by
$$
n=\frac{1}{\sqrt{1+|\nabla h|^2}} \begin{pmatrix} -\nabla h \\ 1 \end{pmatrix}.
$$

Notice that the velocity field $v$ is a gradient, that is $v=-\nabla_{x,y}\phi$ 
where $\phi=P+y$ (recall that we set $g=1$). Since $\cnx_{x,y}v=0$, the potential $\phi$ is harmonic, and hence it is fully determined by its trace on the boundary, which is $h$ since $P$ vanishes on the boundary. We have
\be\label{HS4}
\Delta_{x,y}\phi=0\quad \text{in }\Omega,\qquad \phi\arrowvert_{y=h}=h.
\ee
Consequently, $v$ is fully determined by $h$ which implies that 
the Hele-Shaw problem 
simplifies to an evolution equation for $h$ only; 
namely the equation~\e{HS3}. Once $h$ is determined, one obtains $\phi$ by solving~\e{HS4} and then one sets $v=-\nabla_{x,y}\phi$ and $P=-\phi-y$.

The previous reduction to an evolution equation for $h$ is better formulated by introducing 
the Dirichlet-to-Neumann operator (this operator plays a key role in the analysis by Walter Craig and 
Catherine Sulem of the water-waves equations). 
For a given time~$t$, that is omitted here, 
and a function $\psi=\psi(x)$, 
$G(h)\psi$ is defined by (see~\S\ref{S:21} for details)
$$
G(h)\psi (x)=\sqrt{1+|\nabla h|^2}\partial_n\varphi\arrowvert_{y=h(x)}
=\partial_y\varphi(x,h(x))-\nabla h(x)\cdot\nabla\varphi(x,h(x)),
$$
where $\varphi$ is the harmonic extension of $\psi$, given by
\be\label{defi:varphi}
\Delta_{x,y}\varphi=0\quad \text{in }\Omega,\qquad \varphi\arrowvert_{y=h}=\psi.
\ee
Then, with this notation, it follows from 
the equation~\e{HS3} that (see~\S\ref{S:22}) 
\begin{equation}\label{n7}
\partial_{t}h+G(h)h=0.
\end{equation}
This equation is analogous to the Craig--Sulem--Zakharov formulation of the water-waves equations (following Zakharov~\cite{Zakharov1968} and Craig--Sulem~\cite{CrSu}).

There are many other possible approaches to study the Cauchy problem for the 
Hele-Shaw equation. One can study the existence of weak solutions, viscosity solutions or classical solutons; we refer the reader to \cite{ChangLaraGuillenSchwab,Chen-ARMA-1993,Cheng-Belinchon-Shkoller-AdvMath,CCG-Annals,Escher-Simonett-ADE-1997,GG-JPS-AdvMaths-2019,Gunther-Prokert-SIAM-2006,Hadzic-Shkoller-CPAM2015,Kim-ARMA2003,Knupfer-Masmoudi-ARMA-2015,Pruss-Simonett-book}. These papers consider different formulations of the Hele-Shaw 
problem and we notice that,  for rough solutions, it is not obvious to check that these formulations are equivalent. In this 
article, we are interested in proving some qualitative properties of the flow. To do so, we consider classical solutions (in the sense 
of Definition~\ref{defi:regular} below). The parabolic smoothing effect implies that, for positive times, 
these solutions are 
$C^\infty$ in space and time so that it is elementary to rigorously justify the computations.

\subsection{Main results} In this paper we study some properties of the Hele-Shaw equation which are related to convexity. 
Firstly, we study the existence of Lyapunov functions of the form
$$
I_\Phi(t)=\int_{\xT^n}\Phi(h(t,x))\dx.
$$
We show that if both $\Phi$ and $\Phi'$ are convex, then $I_\Phi(t)$ is a strong Lyapunov function, by this we mean that $t\mapsto I_\Phi(t)$ 
is a non-increasing {\em convex}  function. To study this problem, 
we will introduce a new elliptic formulation of the Hele-Shaw equation. Namely, we observe that the linearized Hele-Shaw equation can be written as $\Delta_{t,x}h=0$ and find an analogous elliptic formulation equation for the nonlinear Hele-Shaw equation. 
Eventually, we study the role of convexity by seeking entropy-type inequalities. 

\textbf{Lyapunov functions.} 
Consider a convex function $\Phi\colon \xR\to\xR^+$. 
With Nicolas Meunier and Didier Smets 
we proved in~\cite{AMS} that
$$
I_\Phi\colon [0,T]\to \xR^+,\quad  t\mapsto \int \Phi(h(t,x))\dx
$$
is a Lyapunov function (which means that the latter quantity is a non-increasing positive function). 
The first main result of this paper is that, if one further assumes that 
the derivative $\Phi'$ is also convex, then the latter quantity is a 
{\em strong Lyapunov function}; by this we mean that it is a non-increasing convex function. 

\begin{theorem}\label{T1}
Consider a smooth solution $h$ to the Hele-Shaw equation. 

$i)$ If $\Phi\colon\xR\to\xR^+$ is a 
$C^2$ convex function, then
\be\label{convexity1}
\fract I_\Phi\le 0 \quad\text{where}\quad I_\Phi(t)=\int_{\xT^n}\Phi(h(t,x))\dx. 
\ee
$ii)$ Assume that $\Phi\colon\xR\to\xR^+$ is a 
$C^3$ convex function whose derivative is also convex. 
Then
\be\label{convexity2}
\fract I_\Phi\le 0\quad\text{and}\quad\fractt I_\Phi\ge 0.
\ee
\end{theorem}
\begin{remark}\begin{enumerate}[i)]
\item In \cite{AMS}, the inequality~\e{convexity1} is proved only 
for $\Phi(h)=h^{2p}$ for all $p$ in $\{1\}\cup 2\xN$; but the generalization to an 
arbitrary convex function is straightforward.

\item To the author's knowledge, the study of the existence of strong convex Lyapunov function is new. 
\item It follows from Stokes' theorem that $\int_{\xT^n} h G(h)h\dx\ge 0$ (see~\e{positivityDN}) . So,  
by multiplying the equation $\partial_t h+G(h)h=0$ by $h$ and integrating over $\xT^n$, one obtains the classical 
result that the $L^2$-norm is a Lyapunov function:
$$
\fract \int_{\xT^n}h(t,x)^2\dx\le 0.
$$
This is the special case for~\e{convexity1} with $\Phi(h)=h^2$. On the other hand, 
the fact that~\e{convexity2} holds for $\Phi(h)=h^2$ is already highly non trivial. 
Indeed, this follows from the 
following identity (first proved in~\cite{AMS}):
$$
\fractt \int_{\xT^n}h^2\dx=-\fract\int_{\xT^n}hG(h)h\dx=\int_{\xT^n}a \la \nabla_{t,x}h\ra^2\dx\ge 0,
$$
where $a$ is a positive coefficient (this is the so-called Taylor coefficient).
\end{enumerate}
\end{remark}

\textbf{An elliptic formulation.} 
To prove Theorem~\ref{T1}, 
we will introduce 
an elliptic formulation of the Hele-Shaw problem. To explain this, we begin by considering 
the linearized equation $\partial_t h+G(0)h=0$. Recall that the Dirichlet-to-Neumann operator $G(0)$ associated to a flat half-space is given explicitly by $G(0)=\lvert D_x\rvert$, that is the Fourier multiplier defined by 
$\lvert D_x\rvert e^{ix\cdot \xi}
=\lvert \xi\rvert e^{ix\cdot \xi}$. 
Then the linearized Hele-Shaw equation reads
$$
\partial_t h+\la D_x\ra h=0.
$$
Now, observe that the previous equation is 
elliptic. Indeed, its symbol $i\tau+\la \xi\ra$ is obviously an elliptic symbol or order $1$. Another way to see this is to make act 
$\partial_t-\la D_x\ra$ on the equation. Since $-\la D_x\ra^2=\Delta_x$, we find
$$
\Delta_{t,x}h=\partial_t^2 h+\Delta_x h=0.
$$
The next result generalizes this observation to the Hele-Shaw equation.

\begin{theorem}\label{proposition:elliptic}
If $h$ is a smooth solution to $\partial_t h+G(h)h=0$ then 
$$
\Delta_{t,x}h+B(h)^*\big( \la \nabla_{t,x}h\ra^2\big)=0,
$$
where $B(h)^*$ is the adjoint (for the $L^2(\xT^n)$-scalar product) of the operator defined by
$$
B(h)\psi=\partial_y \varphi\arrowvert_{y=h},
$$
where $\varphi$ is the harmonic extension of $\psi$ (given by~\e{defi:varphi}).
\end{theorem}

\textbf{An entropy inequality.} 
Then we study the role of convexity to control 
the spatial derivatives of the solutions. We consider the 
Rayleigh--Taylor coefficient $a$, which is a positive function defined by 
$a=-(\partial_y P)\ah$.
It is known that this coefficient is always positive 
when the free surface is at least 
$C^{1,\alpha}$ for some $\alpha>0$ (see~\cite[Prop.\ 4.3]{AMS}). As a consequence, we may consider 
$\sqrt{a}$ and $\log(a)$. Inspired by the study of entropies for elliptic or parabolic equations, we consider the convex function $\varphi(x)=x\log x$ and find 
that $\varphi(1/\sqrt{a})$ is a sub-solution of a well-posed equation. 

\begin{proposition}\label{P:entropy} 
Introduce the operator $L(h)$ defined by
$$
L(h)f=-V\cdot \nabla f-\mez (\cnx V)f
+\sqrt{a}G(h)\big(\sqrt{a}f\big).
$$
The function 
$$
v\defn \frac{1}{\sqrt{a}}\log\left(\frac{1}{\sqrt{a}}\right)
$$
satisfies
\be\label{ineg:aloga}
\partial_t v+L(h)v+cv=f,
\ee
where $f(t,x)\le 0$ and $c=c(t,x)\ge 0$.
\end{proposition}
\begin{remark} 
Observe that $L(h)$ is a non-negative 
operator. For any function $f$, it follows from the inequality~\e{positivityDN} below that
$$
\int_{\xT^n} fL(h)f\dx=\int (\sqrt{a}f) G(h)(\sqrt{a}f)\dx\ge 0.
$$
\end{remark}
The main interest of the previous result 
lies in the fact that it was surprising 
to find an equation involving derivatives of the unknown where both $c$ and $f$ have favorable signs 
(for other candidates, one obtains equations of the form~\e{ineg:aloga} 
where either $f$ has no sign or $c\le 0$).  
As an application of the previous entropy inequality, we 
will give an alternate proof of the following 
result first proved in~\cite{AMS}.

\begin{corollary}\label{T:Bp}
Let $n\ge 1$ and consider a regular solution $h$ to 
the Hele-Shaw equation defined on $[0,T]$. 
Then, for all time $t$ in $[0,T]$,
$$
\inf_{x\in\xT^n}a(t,x)\ge \inf_{x\in \xT^n}a(0,x).
$$
\end{corollary}

\textbf{To Walter.} 
With Guy M\'etivier (\cite{AM}), we started working on the water-waves equations and the Dirichlet-to-Neumann operator by reading 
a very well-written paper, in French, by Walter Craig and Ana-Maria Matei (\cite{CM}). 
Over the years, I met Walter frequently during conferences, in Canada 
or during his visits in France. 
He was always generous with his ideas. His original points of view, 
his enthusiasm and 
his questions deeply influenced me. I wish I could thank him 
one more time for all he did to help me. 

\section{Preliminaries} 
In this section we review several results about the Dirichlet-to-Neumann operator 
as well as 
some identities proved in~\cite{AMS} about the Hele-Shaw equation.

\subsection{The Dirichlet-to-Neumann operator}\label{S:21}
In this paragraph the time variable is seen as a parameter and we skip it. 
We denote by $H^s(\xT^n)$ the Sobolev space of periodic functions $u$ 
such that $(I-\Delta)^{s/2}u$ belongs to 
$L^2(\xT^n)$, where $(I-\Delta)^{s/2}$ is the Fourier multiplier with symbol $(1+\la\xi\ra^2)^{s/2}$. 

Now consider a smooth function $h \in C^{\infty}(\xT^n)$ 
and a function $\psi$ in the Sobolev 
space $H^{\mez}(\xT^n)$. 
Then it follows from classical arguments that there is a 
unique variational solution 
$\varphi$ to the problem
\be\label{defi:phi}
\Delta_{x,y}\varphi=0 \quad\text{in }\Omega=\{y<h(x)\},\quad 
\varphi\arrowvert_{y=h}=\psi.
\ee
Notice that $\nabla_{x,y}\varphi$ belongs only to $L^2(\Omega)$, 
so it is not obvious that one can consider the trace 
$\partial_n \varphi\arrowvert_{\partial\Omega}$. However, 
since $\Delta_{x,y}\varphi=0$, one can express the normal 
derivative in terms of the tangential derivatives and 
$\sqrt{1+|\nabla h|^2}\partial_n \varphi\arrowvert_{\partial\Omega}$ 
is well-defined and belongs to $H^{-\mez}(\xT^n)$. 
As a result, one can define the Dirichlet-to-Neumann operator $G(h)$ by
$$
G(h)\psi (x)=\sqrt{1+|\nabla h|^2}\partial_n\varphi\arrowvert_{y=h(x)}
=\partial_y\varphi(x,h(x))-\nabla h(x)\cdot\nabla\varphi(x,h(x)).
$$
Let us recall two results. Firstly, 
it follows from classical elliptic regularity results that, for any 
$s\ge 1/2$, $G(h)$ is bounded from $H^s(\xT^n)$ into 
$H^{s-1}(\xT^n)$. This property still holds in 
the case where $h$ has limited regularity. Many results have 
been obtained since the pioneering works of 
Craig and Nicholls (\cite{CN}; see also  \cite{WuInvent,Gunther-Prokert-SIAM-2006,LannesJAMS}). 
It is known that (see~\cite{ABZ3,LannesLivre}), for any $s>n/2+1$,
\be\label{DN:Sobolev}
\lA G(h)\psi\rA_{H^{s-1}}\le C\big(\lA h\rA_{H^{s}}\big)\lA \psi\rA_{H^{s}}.
\ee
Secondly, we will frequently use the fact that $G(h)$ is a positive operator. 
Namely, 
consider a function $\psi=\psi(x)$ and its 
harmonic extension~$\varphi=\varphi(x,y)$, solution to~\e{defi:phi}. It follows from Stokes' theorem that
\be\label{positivityDN}
\int_{\xT^n} \psi G(h)\psi\dx=\int_{\partial\Omega}\varphi \partial_n \varphi\dsigma=
\iint_{\Omega}\la\nabla_{x,y}\varphi\ra^2\dydx\ge 0.
\ee

In addition to the Dirichlet-to-Neumann operator, 
we will use the operators $B(h), V(h)$ defined by
\begin{align*}
B(h)\psi&=\partial_y \varphi\ah,\\
V(h)\psi&=(\nabla_x \varphi)\ah,
\end{align*}
where again $\varphi$ is the harmonic extension of $\psi$ given by \e{defi:phi}. 

We recall the following identities. 

\begin{lemma}\label{L:31}
We have
\be\label{n897}
B(h)\psi=\frac{G(h)\psi+\nabla h\cdot\nabla \psi}{1+\la \nabla h\ra^2},\qquad 
V(h)\psi=\nabla \psi-(B(h)\psi) \nabla h,
\ee
and
\be\label{n991}
G(h)B(h)\psi=-\cnx V(h)\psi.
\ee
\end{lemma}
\begin{proof}
By definition of the operator $G(h)$,
\be\label{n997}
G(h)h=\big( \partial_y \varphi-\nabla h\cdot \nabla \varphi\big)\ah=B(h)\psi-\nabla h\cdot 
V(h)\psi. 
\ee
On the other hand, it follows from the chain rule that
$$
\nabla \psi=\nabla_x (\varphi\ah)=(\nabla_x \varphi\ah)
+(\partial_y\varphi)\ah \nabla h=V(h)\psi+(B(h)\psi)\nabla h. 
$$
Consequently, we obtain the wanted identity for $V(h)\psi$:
$$
V(h)\psi=\nabla \psi -(B(h)\psi)\nabla h.
$$
Now, by reporting this formula in \e{n997} we get
$$
G(h)h=(1+\la \nabla h\ra^2)B(h)\psi-\nabla\psi \cdot\nabla h,
$$
which immediately implies the desired result for $B(h)\psi$. 

The identity \e{n991} is proved in~\cite{ABZ3,BLS,LannesJAMS}, 
see also Proposition~5.1 in~\cite{AMS}.
\end{proof}

\subsection{A reformulation}\label{S:22}
In this paragraph, we give more details about 
the formulation of the Hele-Shaw equation in terms of 
the Dirichlet-to-Neumann operator given in the introduction. 

The Dirichlet-to-Neumann operator 
plays a key role in the study of the water-waves problem 
since the seminal works of Zakharov~\cite{Zakharov1968} and Craig and Sulem~\cite{CrSu}. 
It enters also in a very natural way in the study of the Hele-Shaw equation. Recall from the introduction that
$$
v=-\nabla_{x,y}\phi\quad\text{with}\quad \phi=P+y.
$$
Since $\cnx_{x,y}v=0$ and since $P\arrowvert_{y=h}=0$, 
the potential $\phi$ satisfies
$$ 
\Delta_{x,y} \phi=0,\qquad \phi\arrowvert_{y=h}=h.
$$
We conclude that $\phi$ is the harmonic extension of $g h$, which implies that
$$
\sqrt{1+|\partialx h|^2}\, v\cdot n=-G(h)h.
$$
Consequently, the evolution equation for $h$ simplifies to
\begin{equation}\label{n10}
\partial_{t}h+G(h)h=0.
\end{equation}
Recall from \e{DN:Sobolev} that $G(h)h$ is well-defined whenever $h$ 
takes values in $H^s(\xT^n)$ for some $s>n/2+1$. 
The following result allows to solve the Cauchy problem in this general setting.

\begin{theorem}[from \cite{AMS,NPausader}]\label{T:Cauchy}
Let $n\ge 1$ and consider a real number $s>n/2+1$. 
For any initial data $h_0$ in $H^s(\xT^n)$, there exists 
a time $T>0$ such that the Cauchy problem
\begin{equation}\label{Hele-Shaw100}
\partial_{t}h+G(h)h=0,\quad h\arrowvert_{t=0}=h_0,
\end{equation}
has a unique solution satisfying
$$
h \in C^0([0,T];H^s(\xT^n))\cap C^1([0,T];H^{s-1}(\xT^n))\cap L^2([0,T];H^{s+\mez}(\xT^n)).
$$
Morevoer, $h$ belongs to $C^\infty((0,T]\times \xT^n)$. 
\end{theorem}
\begin{definition}\label{defi:regular}
We say that $h$ is a regular solution to \e{Hele-Shaw100} 
if $h$ satisfies the conclusions of the above result on some time interval $[0,T]$.
\end{definition}

\subsection{Equations for the derivatives}

As shown in \cite{AMS}, it is very convenient to work with some special derivatives of the solutions. 
Guided by the analysis in~\cite{ABZ3}, we introduce the horizontal and vertical traces 
of the velocity at the free surface:
\be\label{defi:BVphii}
B= (\partial_y \phi)\arrowvert_{y=h},\quad 
V = (\nabla_x \phi)\arrowvert_{y=h}.
\ee
They are given in terms of $h$ by the following formulas (see Lemma~\ref{L:31}),
\begin{equation}\label{n1201}
B= \frac{G(h)h+\la \partialx h\ra^2}{1+|\partialx  h|^2},
\qquad
V=(1-B)\partialx h.
\end{equation}
We also introduce the Rayleigh--Taylor coefficient $a$ defined by
\be\label{defi:RT}
a=-(\partial_y P)\ah=1-B.
\ee

There are two important positivity results which follow 
from the maximum principle (or the Hopf-Zaremba's principle). 
The first one is the well-known positivity of the Taylor coefficient 
(see~\cite[Prop.\ 4.3]{AMS}). 

\begin{proposition}\label{Coro:Zaremba-Taylor}
For any regular solution $h$ to the Hele-Shaw equation, there holds
$$
a=1-B>0.
$$
\end{proposition}

The next results gives an evolution equation for $B$ and contains a positivity results for a 
coefficient $\gamma$. 
\begin{proposition}[see Prop.\ 5.2 in \cite{AMS}]\label{P:25}
Assume that $h$ is a regular solution to the Hele-Shaw equation. Then $B$ satisfies
\be\label{defi:gamma}
\partial_t B-V\cdot \nabla B +aG(h)B=\gamma,
\ee
where
\be\label{ngammadef}
\gamma=\frac{1}{1+|\nabla h|^2}\Big(G(h)\big(B^2+|V|^2\big)-2BG(h)B-2V\cdot G(h)V\Big).
\ee
Moreover, the coefficient $\gamma$ satisfies
\be\label{ngammasign}
\gamma\le 0.
\ee
\end{proposition}

\subsection{Shape derivatives}
Notice that $G(h)\psi$ is linear in $\psi$ but depends nonlinearly in $h$. 
This is one of the main difficulty to study the Hele-Shaw equation. The same problem appears for 
the water-waves problem. One tool to study the dependence in $h$ is 
to consider the shape derivative formula, as given by the following
\begin{proposition}[from Lannes~\cite{LannesJAMS,LannesLivre}]\label{P:shape}
Consider a real number $s$ such that $s>1+n/2$. 
Let $\psi \in H^s(\xT^n)$ and $h\in H^{s} (\xT^n)$. 
Then there is a neighborhood $\mathcal{U}_h\subset H^{s}(\xT^n)$ of $h$ such that
the mapping
$$
h \in \mathcal{U}_h \mapsto G(h)\psi \in H^{s-1}(\xT^n)
$$
is differentiable. Moreover, for all $\zeta\in H^s(\xT^n)$, we have 
\be\label{n:shape}
d G(h)\psi \cdot  \zeta \defn
\lim_{\eps\rightarrow 0} \frac{1}{\eps}\big\{ G(h+\eps \zeta)\psi -G(h)\psi\big\}
= -G(h)(\mathfrak{B}\zeta) -\cnx (\mathfrak{V}\zeta),
\ee
where
$$
\mathfrak{B}=\frac{G(h)\psi+\partialx h\cdot\partialx\psi}{1+\la \partialx h\ra^2},
\quad \mathfrak{V} =\partialx\psi-\mathfrak{B}\partialx h.
$$
\end{proposition}
This result is proved for smoother function by Lannes in~\cite{LannesJAMS}. 
We refer to his monograph~\cite{LannesLivre} for the proof 
in the general case. However, in this paper, to justify the computations, we only 
need this result for smooth functions.

\section{Elliptic formulation}

In this section we prove Theorem~\ref{proposition:elliptic}. 

Let us recall (see~\S\ref{S:21}) that the operators $B(h)$ and $V(h)$ are 
given by
$$
B(h)\psi=\partial_y \phi\ah,\qquad
V(h)\psi=(\nabla_x \phi)\ah,
$$
where $\phi$ is the harmonic extension of $\psi$ given by \e{defi:phi}. 
Denote by $B(h)^*$ the adjoint for the $L^2(\xT^n)$-scalar product. In light of the identity~\e{n897}, 
one has
\be\label{n998}
B(h)^*\psi=G(h)\left( \frac{\psi}{1+\la \nabla h\ra^2}\right)-\cnx \left(\frac{\psi}{1+\la \nabla h\ra^2}\nabla h\right).
\ee
We are now ready to prove Theorem~\ref{proposition:elliptic} whose statement is recall here.
\begin{theorem}\label{proposition:elliptic-bis}
If $h$ is a smooth solution to $\partial_t h+G(h)h=0$ then 
\be\label{n999}
\Delta_{t,x}h+B(h)^*\big( \la \nabla_{t,x}h\ra^2\big)=0.
\ee
\end{theorem}
\begin{proof}
The proof is in two steps. 
We begin by differentiating in time the Hele-Shaw equation
$$
\partial_t h+G(h)h=0. 
$$
It follows from the shape derivative formula~\e{n:shape} that
$$
\partial_t^2 h=-\partial_t G(h)h=-G(h)\big((1-B)\partial_t h\big)+\cnx \big((\partial_{t}h) V\big),
$$
where 
\be\label{n990}
B=B(h)h=\frac{G(h)h+\nabla h\cdot\nabla h}{1+\la \nabla h\ra^2},\qquad 
V=V(h)h=(1-B)\nabla h.
\ee
We next compute $(1-B)\partial_t h$. To do so, we replace $\partial_t h$ by $-G(h)h$ and observe that, 
by definition of the operator $G(h)$,
$$
G(h)h=\big( \partial_y \phi-\nabla h\cdot \nabla \phi\big)\ah=B-V\cdot \nabla h. 
$$
Recalling that $V=(1-B)\nabla h$, it follows that 
\begin{align*}
(1-B)\partial_t h&=-(1-B)G(h)h\\
&=-(1-B)(B-V\cdot\nabla h)\\
&=-B+B^2+V\cdot \big((1-B)\nabla h\big)\\
&=-B+B^2+\la V\ra^2.
\end{align*}
The previous results yield
$$
\partial_t^2 h-G(h)B+G(h)\big(B^2+\la V\ra^2\big)+\cnx \big((B-V\cdot\nabla h)V\big)=0.
$$

We then use the identity $G(h)B=-\cnx V$ (see~\e{n991}) to infer that
$$
\partial_t^2 h+\cnx V +G(h)\big(B^2+\la V\ra^2\big)+\cnx \big((B-V\cdot\nabla h)V \big)=0.
$$
So, replacing $V$ by $\nabla h-B\nabla h$ in $\cnx V$, we have
\be\label{n1000}
\partial_t^2 h+\Delta_x h-\cnx (B\nabla h)
+G(h)\big(B^2+\la V\ra^2\big)+\cnx\big((B-V\cdot\nabla h)V\big)=0.
\ee
To simplify this expression, we begin by observing that
\be\label{n1010}
-\cnx (B\nabla h)+\cnx\big((B-V\cdot\nabla h)V\big)=\cnx\Big(B(V-\nabla h)-(V\cdot\nabla h)V\Big).
\ee
Now
$$
V-\nabla h=-B\nabla h,
$$
so the first term in the right-hand side of~\e{n1010} can be written as
$$
\cnx(B(V-\nabla h))=-\cnx (B^2\nabla h).
$$
Moving to the second term in the right-hand side of \e{n1010}, using again $V=(1-B)\nabla h$, 
we verify that
\begin{align*}
(V\cdot\nabla h)V&=((1-B)\nabla h\cdot \nabla h)(1-B)\nabla h\\
&=(1-B)^2\la \nabla h\ra^2\nabla h\\
&=\la V\ra^2\nabla h.
\end{align*}
Consequently, the identity~\e{n1000} simplifies to
\be\label{n1001}
\Delta_{t,x}h+G(h)(B^2+\la V\ra^2)-\cnx \big((B^2+\la V\ra^2)\nabla h\big)=0.
\ee

The next step consists in expressing $B^2+\la V\ra^2$ in terms of $\nabla_{t,x}h$. 
To do so, using the identities in~\e{n990}, we verify that
\begin{align*}
B^2+\la V\ra^2&=B^2+(1-B)^2\la\nabla h\ra^2\\
&=\left(\frac{G(h)h+\la \nabla h\ra^2}{1+\la \nabla h\ra^2}\right)^2+\left(\frac{1-G(h)h}{1+\la \nabla h\ra^2}\right)^2\la\nabla h\ra^2\\
&=\frac{(G(h)h)^2+\la \nabla h\ra^2}{1+\la \nabla h\ra^2}. 
\end{align*}
Since $\partial_t h=-G(h)h$, we conclude that
\be\label{n1002}
B^2+\la V\ra^2=\frac{(\partial_t h)^2+\la \nabla h\ra^2}{1+\la \nabla h\ra^2}
=\frac{\la \nabla_{t,x}h\ra^2}{1+\la \nabla h\ra^2}. 
\ee
Therefore, the wanted identity~\e{n999} follows from \e{n1001}, \e{n1002} 
and the definition~\e{n998} of~$B(h)^*$.
\end{proof}

\section{Lyapunov functionnals}

In this section, we prove Theorem~\ref{T1}.

\begin{lemma}\label{proposition:Lyapunov1}
Consider a smooth solution $h$ to the Hele-Shaw equation. 
If $\Phi\colon\xR\to\xR$ is a 
$C^2$ convex function, then
\be\label{convexity3}
\fract I_\Phi\le 0 \quad\text{where}\quad I_\Phi(t)=\int_{\xT^n}\Phi(h(t,x))\dx. 
\ee
\end{lemma}
\begin{proof}We follow the analysis in~\cite{AMS}. 
In~\cite{CC-PNAS-2003,CC-CMP-2004},  C{\'o}rdoba and C{\'o}rdoba proved  
that, for any exponent $\alpha$ in $[0,1]$ and 
any $C^2$ function $f$ decaying sufficiently fast at infinity, 
one has the pointwise inequality
$$
2 f (-\Delta)^\alpha f \ge (-\Delta)^\alpha (f^2).
$$
This inequality has been generalized and applied to many different problems 
(see~\cite{Ju2005maximum,ConstantinIgnatova1,ConstantinIgnatova2,ConstantinTV} 
and the numerous references there in). 
Recently, C\'ordoba and Mart\'\i nez (\cite{CordobaM}) proved that
\be\label{n1105}
\Phi'(f)G(h)f\ge G(h)\big(\Phi(f)\big),
\ee
when $h$ is a $C^2$ function and 
$\Phi(f)=f^{2m}$ for some positive integer $m$. 
In~\cite{AMS}, this result is generalized to the case where $\Phi\colon \xR\to\xR$ 
is an arbitrary $C^2$ convex function and $f,h$ belong 
to some H\"older space $C^{1,\alpha}(\xT^n)$ with $\alpha>0$. 
By using the latter result, we immediately obtain \e{convexity3}. Indeed, by multiplying 
the Hele-Shaw equation by $\Phi'(h)$ and integrating over $\xT^n$, we get that
$$
\fract \int\Phi(h)\dx+\int \Phi'(h)G(h)h\dx=0.
$$
Now, we use the fact that $\int G(h)\psi\dx =0$ for any function $\psi$ to deduce from~\e{n1105} that
\be\label{conv5}
\int \Phi'(h)G(h)h\dx\ge \int G(h)\Phi(h)\dx=0.
\ee
This completes the proof of~\e{convexity1}.
\end{proof}

We now prove the main result. 
\begin{lemma}\label{proposition:Lyapunov2}
Consider a smooth solution $h$ to the Hele-Shaw equation. 
If $\Phi\colon\xR\to\xR$ is a 
$C^3$ convex function whose derivative $\Phi'$ is also convex, then
\be\label{convexity4}
\fractt I_\Phi\ge 0.
\ee
\end{lemma}
\begin{proof}
We have seen in the previous proof that
$$
\fract I_\Phi+\int \Phi'(h)G(h)h\dx=0.
$$
So it is sufficient to show that
\be\label{n1106}
\fract \int \Phi'(h)G(h)h\dx\le 0.
\ee
Notice that the latter result is interesting in itself since it asserts 
that
$$
\int \Phi'(h)G(h)h\dx
$$
is a Lyapunov functionnal (this is indeed a coercive quantity, see~\e{conv5}).

To prove~\e{n1106}, we use the elliptic formulation of the Hele-Shaw equation 
given by Theorem~\ref{proposition:elliptic}. Recall that
$$
-\Delta_{t,x}h-B(h)^*\big(\la\nabla_{t,x}h\ra^2\big)=0.
$$
We multiply this equation by $\Phi'(h)$ and integrate first in space. 
This gives that
$$
-\int \Phi'(h)\partial_t^2 h\dx +\int \Phi''(h)\la \nabla_x h\ra^2\dx -\int \big(B(h)\Phi'(h)\big)
\la\nabla_{t,x}h\ra^2\dx=0.
$$
It follows from the identity~\e{n897} for the operator $B(h)$ that
$$
B(h)\Phi'(h)=\frac{G(h)\Phi'(h)+\nabla h\cdot\nabla \Phi'(h)}{1+\la \nabla h\ra^2}.
$$
Since $\Phi'(h)$ is convex, the inequality~\e{n1105} implies that
$$
G(h)\Phi'(h)\le \Phi''(h)G(h)h.
$$
It follows that
$$
B(h)\Phi'(h)\le \Phi''(h)\frac{G(h)h+\la\nabla h\ra^2}{1+\la\nabla h\ra^2}=\Phi''(h)B \quad\text{where }
B=\frac{G(h)h+\la\nabla h\ra^2}{1+\la\nabla h\ra^2}.
$$
Consequently,
\be\label{n1109}
-\int \Phi'(h)\partial_t^2 h\dx +\int \Phi''(h)\la \nabla_x h\ra^2\dx -\int  \Phi''(h)B \la\nabla_{t,x}h\ra^2\dx\le 0.
\ee
Now consider a time $T>0$ and 
integrate by parts in time on $[0,T]$ to obtain
$$
\int_0^T\int \Phi'(h)\partial_t^2 h\dx\dt
=\int \Phi'(h)\partial_t h\dx \bigg\arrowvert_{t=0}^{t=T}-
\int_0^T\int \Phi''(h)(\partial_t h)^2\dx\dt.
$$
By combining this with~\e{n1109}, we find that
$$
-\int \Phi'(h)\partial_t h\dx \bigg\arrowvert_{t=0}^{t=T}
+\int_0^T\int \Phi''(h)(1-B) \la \nabla_{t,x} h\ra^2\dx\dt\le 0.
$$
Remembering that $a=1-B$ and $\partial_t h=-G(h)h$, the preceding inequality implies that
\begin{multline*}
\int \Phi'(h)G(h)h\dx \bigg\arrowvert_{t=T}
+\int_0^T\int \Phi''(h)a \la \nabla_{t,x} h\ra^2\dx\dt\\
\le
\int \Phi'(h)G(h)h\dx \bigg\arrowvert_{t=0}.
\end{multline*}
We now use the fact that the Taylor coefficient $a$ 
is positive 
(see~Proposition~\ref{Coro:Zaremba-Taylor}) 
and the fact that $\Phi$ is convex 
to deduce that $a\Phi''(h)\ge 0$. 
This concludes the proof of~\e{n1106} and hence the proof of the lemma.
\end{proof}

\section{Convexity and entropy}

Here we prove Proposition~\ref{P:entropy} and its corollary. 
Recall the notation
$$
L(h)f=-V\cdot \nabla f-\mez (\cnx V)f
+\sqrt{a}\, G(h)\big(\sqrt{a}f\big).
$$
Recall also that $a(t,x)>0$ for all $t,x$, so that one may consider $\sqrt{a}$ and $\log(a)$.
\begin{proposition}
For any positive constant $m>0$, the function 
$$
u\defn \frac{\log (ma)}{\sqrt{a}},
$$
satisfies
\be\label{n1189}
\partial_t u+L(h)u- \frac{\gamma}{2a}u\ge 0.
\ee
\end{proposition}
\begin{remark}
$i)$ With $m=1$, we have
$$
u=-2v \quad\text{where}\quad v=\frac{1}{\sqrt{a}}\log\left(\frac{1}{\sqrt{a}}\right).
$$
So the preceding proposition implies the result of Proposition~\ref{P:entropy} with
$$
c=-\frac{\gamma}{2a}.
$$
Since $\gamma\le 0$, the latter function is non-negative.

$ii)$ Notice that the right-hand side in~\e{n1189} does not depend on $m$.

$iii)$ We use the parameter $m$ below to control $\inf_{x}a(t,x)$.
\end{remark}
\begin{proof}
Assume that $h$ is a regular solution to the Hele-Shaw equation. As recalled in Proposition~\ref{P:25}, 
the function $B$ satisfies
$$
\partial_t B-V\cdot \nabla B +aG(h)B=\gamma,
$$
where $\gamma\le 0$ is given by
$$
\gamma=\frac{1}{1+|\nabla h|^2}\Big(G(h)\big(B^2+|V|^2\big)-2BG(h)B-2V\cdot G(h)V\Big).
$$
Since $a=1-B$, using the fact that $G(h)1=0$, 
we deduce that
\be\label{n1201b}
\partial_t a-V\cdot \nabla a +aG(h)a+\gamma=0,
\ee
together with
$$
\gamma=\frac{1}{1+|\nabla h|^2}\Big(G(h)\big(a^2+|V|^2\big)-2aG(h)a-2V\cdot G(h)V\Big).
$$

Since $a$ is a positive function, we may 
multiply the equation~\e{n1201b}Ê
by $1/a$, to obtain at once
\be\label{n1203}
(\partial_t -V\cdot \nabla)\log a +G(h)a+\frac{\gamma}{a}=0.
\ee
We now claim that
\be\label{n1205}
G(h)a\le a \, G(h)\log a.
\ee
To do so, we use the fact that $\log$ 
is a concave function and the fact that $a$ 
is bounded from below 
by a positive constant $c_0>0$ on $[0,T]\times \xT^n$. 
This allows us to consider a smooth concave function 
$\theta\colon\xR\to\xR$ which coincides with $\log$ on $[c_0/2,+\infty)$. 
As a result, the inequality~\e{n1105} 
implies that
$$
G(h)\log(a)=G(h)\theta(a)\ge \theta'(a)G(h)a=\frac{1}{a}G(h)a,
$$
which in turn implies \e{n1205}. We next apply \e{n1205} to deduce 
from~\e{n1203} that
$$
(\partial_t -V\cdot \nabla)\log a
+a\, G(h)\log a+\frac{\gamma}{a}\ge 0.
$$
Since $G(h)C$ vanishes for any constant $C$, the preceding inequality 
implies that, for any positive constant $m>0$,
\be\label{n1203b}
(\partial_t -V\cdot \nabla)\log (ma)
+a\, G(h)\log (ma)+\frac{\gamma}{a}\ge 0.
\ee

Now we observe that
\begin{align*}
(\partial_t -V\cdot \nabla)\frac{1}{\sqrt{a}}
&=-\mez \frac{(\partial_t -V\cdot \nabla)a}{a\sqrt{a}}\\
&=\mez\frac{a\, G(h)a+\gamma}{a\sqrt{a}} \qquad (\text{see }\e{n1201})\\
&=\mez \frac{a \cnx V+\gamma}{a\sqrt{a}},
\end{align*}
where we used the identity $G(h)a=-G(h)B=\cnx V$ (see~\e{n991}) in the last line. 
Consequently, 
\begin{align*}
&(\partial_t -V\cdot \nabla)\frac{\log (ma)}{\sqrt{a}}\\
&\qquad\qquad=\frac{1}{\sqrt{a}}(\partial_t -V\cdot \nabla)\log (ma)
+\log (ma)(\partial_t -V\cdot \nabla)\frac{1}{\sqrt{a}}\\
&\qquad\qquad\ge \frac{1}{\sqrt{a}}\Big(-a\, G(h)\log (ma)-\frac{\gamma}{a}\Big)
+\mez \frac{a\, \cnx V+\gamma}{a\sqrt{a}}\log (ma).\\
\end{align*}
Then one easily verifies that 
$u=\log (ma)/\sqrt{a}$ satisfies 
$$
\partial_t u+L(h)u-\mez \frac{\gamma}{a}u\ge -\frac{\gamma}{a\sqrt{a}}.
$$
Since $\gamma\le 0$, this implies the wanted inequality~\e{n1189}.
\end{proof}

We now prove Corollary~\ref{T:Bp} whose 
statement is recalled here.
\begin{corollary}\label{T:Bp2}
Let $n\ge 1$ and consider a regular solution $h$ to 
the Hele-Shaw equation defined on $[0,T]$. 
Then, for all time $t$ in $[0,T]$,
\be\label{n1301}
\inf_{x\in\xT^n}a(t,x)\ge \inf_{x\in \xT^n}a(0,x).
\ee
\end{corollary}
\begin{proof} This result can be proved by exploiting only the fact that $\gamma\le 0$. 
Here, we just want to explain how to recover this from the previous proposition. 

Set
$$
c_0=\inf_{x\in \xT^n}a(0,x),\qquad m=\frac{1}{c_0}. 
$$
Then $m a(0,x)\ge 1$ for all $x\in \xT^n$. Set
$$
u=\frac{\log (ma)}{\sqrt{a}},\quad u_-=\min\{u,0\}.
$$ 
We claim that $u_-=0$. 
This will at once imply that $\log(ma)\ge 0$ so $ma(t,x)\ge 1$ for all 
$(t,x)\in [0,T]\times\xT^n$, which in turn implies $a(t,\cdot)\ge 1/m=c_0$, which is the asserted inequality~\e{n1301}. 

To prove that $u_-=0$, we use Stampacchia's method. 
By multiplying the equation~\eqref{n1189} by $u_-\le 0$, one obtains
$$
\mez \fract \int u_-^2\dx +\int u_-L(h)u\dx -\mez \int \frac{\gamma}{a}u_-^2\dx\le 0.
$$
Now using that $\gamma\le 0$ and $a>0$, we have
$$
\int \frac{\gamma}{a}u_-^2\dx\le 0,
$$
so 
\be\label{n1310}
\mez \fract \int u_-^2\dx +\int u_-L(h)u\dx \le 0.
\ee
On the other hand, proceeding as above, 
the convexity inequality~\e{n1105} applied with the function $x\mapsto x^2 \mathbf{1}_{\xR_-}(x)$ implies that
$$
\int u_-L(h)u\dx=\int \sqrt{a}u_-G(h)(\sqrt{a}u)\dx 
\ge \int G(h)\left(\mez a u_-^2\right)\dx= 0.
$$
As result, the preceding inequality~\e{n1310} simplifies to
$$
\mez \fract \int u_-^2\dx\le 0.
$$
Since $u_-(0,\cdot)=0$ at initial time (by construction), we obtain $u_-(t,\cdot)=0$ for all time $t$, 
which terminates the proof.
\end{proof}

%
%

\vspace{1cm}
\noindent\textbf{Thomas Alazard}\\
\noindent Universit\'e Paris-Saclay, ENS Paris-Saclay, CNRS,\\
Centre Borelli UMR9010, avenue des Sciences, 
F-91190 Gif-sur-Yvette\\

\end{document}